\documentclass[a4paper,10pt]{amsart}

\usepackage[
  margin=30mm,
  marginparwidth=25mm,     
  marginparsep=2mm,       
  bottom=25mm,
  ]{geometry}

\usepackage[latin1]{inputenc}
\usepackage[T1]{fontenc}
\usepackage[english]{babel}
\usepackage{amsmath,amssymb,amsthm}
\usepackage{latexsym}
\usepackage{delarray}
\usepackage{bbm}
\usepackage{hyperref}
\usepackage[pdftex,usenames,dvipsnames]{color}
\usepackage{bbding,trfsigns}
\usepackage{wasysym}
\usepackage{datetime}
\usepackage{mathrsfs}

\usepackage{tikz}

\newtheorem{Th}{Theorem}[section]
\newtheorem{Prop}[Th]{Proposition}
\newtheorem{Lem}[Th]{Lemma}

\newtheorem{Rem}[Th]{Remark}

\newcommand{\R}{\mathbb{R}}
\newcommand{\T}{\mathbb{T}}
\newcommand{\Rn}{\mathbb{R}^{n}}
\newcommand{\Z}{\mathbb{Z}}

\renewcommand{\d}{\partial}

\newcommand{\Ss}{\mathcal{S}}

\newcommand{\PP}{\mathcal{P}}

\newcommand{\SA}{\Ss_A(\Rn)}

\newcommand{\jap}{\langle \xi \rangle}

\newcommand{\brkt}[1]{\left({#1}\right)}

\newcommand{\eps}{\varepsilon}

\DeclareMathOperator{\supp}{supp}

\allowdisplaybreaks 

\title[Transference of local to global maximal estimates]
{Transference of local to global $L^2$ maximal estimates for dispersive Partial differential equations }

\author[A. J. Castro]{Alejandro J. Castro}
\author[S. Rodr\'iguez-L\'opez]{Salvador Rodr\'iguez-L\'opez}
\author[W. Staubach]{Wolfgang Staubach}

\address{\newline
       Alejandro J. Castro \newline
       Department of  Mathematics, Nazarbayev University, \newline
		010000 Astana, Kazakhstan}
\email{alejandro.castilla@nu.edu.kz}

\address{\newline
       Salvador Rodr\'iguez-L\'opez \newline
       Department of Mathematics, Stockholm University,\newline
       SE - 106 91 Stockholm, Sweden}
       \email{s.rodriguez-lopez@math.su.se }

\address{\newline
       Wolfgang Staubach \newline
       Department of  Mathematics, Uppsala University, \newline
       S-751 06 Uppsala, Sweden}
\email{wulf@math.uu.se}

\thanks{
The first author is supported by Nazarbayev University Social Policy Grant and by the Spanish Government grant MTM2016-79436-P.
The second author is partially supported by the Spanish Government grant MTM2016-75196-P.
The third author is partially supported by a grant from the Crafoord foundation and by a grant from G. S. Magnusons fond, grant number MG2015-0077}



\begin{document}


\maketitle

\begin{abstract}
In this paper we give an elementary proof for transference of local to global maximal estimates for dispersive PDEs. This is done by transferring local $L^2$ estimates for certain oscillatory integrals  with rough phase functions, to the corresponding global estimates. The elementary feature of our approach is that it entirely avoids the use of the wave packet techniques which are quite common in this context, and instead is based on scalings and classical oscillatory integral estimates.
\end{abstract}

\section{Introduction}

In the study of the Cauchy problem
\begin{equation}\label{the pde}
\begin{cases}
i\partial_t u(t,x) +\phi(D) u (t,x)=0,\\
u(0,x)= u_0 (x) \in H^{s},\,\,\, \mathrm{for}\,\,\, {s>0},
\end{cases}
\end{equation}
for dispersive equations, oscillatory integral operators of the form
\begin{equation*}
T_t f(x)= \int_{\R^n}  e^{ix\cdot \xi + it\phi(\xi)} \hat{f}(\xi)\, d\xi,
\end{equation*}
 play a crucial role. Here $\phi$ is a positively  homogeneous phase function of degree $a$ that satisfies $|\partial^{\alpha} \phi(\xi)|
	\lesssim |\xi|^{a-|\alpha|}$ outside the origin and $\widehat{\phi(D) u}(t,\xi)= \phi(\xi) \widehat{u}(t,\xi).$ We denote by $H^s$ the usual $L^2$-based Sobolev spaces.\\

In the theory of dispersive partial differential equations it is a classical fact that a local maximal function estimate of the type
\begin{equation}\label{head local maximal estimate}
\Vert \sup_{0< t< 1} |T_t f| \, \Vert_{L^2 (B(0,1))}\leq C\Vert f\Vert_{H^s(\R^n)},
\end{equation}
would imply that the solution $u(x,t)$ of \eqref{the pde} (if it exists) converges pointwise almost everywhere to $u_0$ as $t\to 0$.
The global counterpart of \eqref{head local maximal estimate} i.e.
\begin{equation}\label{main global maximal estimate}
\Vert \sup_{0 <  t <  1} |T_tf| \, \Vert_{L^2 (\R^n)}\leq C\Vert f\Vert_{H^s (\R^n)},
\end{equation}
is also important for the study of the well-posedness of the Cauchy problem \eqref{the pde}.\\


It has been a considerable amount of activity regarding the validity of \eqref{head local maximal estimate} and \eqref{main global maximal estimate} for various dispersive equations. For example one should mention the works of M. Cowling \cite{Cowling}, B. Walther \cite{Walther} in the case of $\phi(\xi)=|\xi|$ (i.e. the wave operator $e^{it\sqrt{-\Delta}}$), papers by P. Sj\"olin \cite{Sjolin0}, \cite{Sjolin1}, \cite{Sjolin2},  \cite{Sjolin-1} concerning $\phi(\xi)=|\xi|^a$, with $a>1$, and the papers by  L. Carleson \cite{Carleson}, L. Vega \cite{Vega1}, S. Lee \cite{Lee} and J. Bourgain \cite{Bourgain}  concerning $\phi(\xi)=|\xi|^2$ (i.e. the Schr\"odinger operator $e^{it\Delta}$). We should also mention the recent result of X. Du, L. Guth and X. Li in \cite{Guth}, where they establish the estimate \eqref{head local maximal estimate} in the range $s>1/3$ for the Schr\"odinger maximal operator in dimension 2. According to a result of Bourgain \cite{Bourgain2}, for the Schr\"odinger operator in $n$ dimensions, \eqref{head local maximal estimate} can be valid only if $s\geq\frac{n}{2(n+1)}$, and so the aforementioned result in \cite{Guth} is sharp up to the end point. For the oscillatory integrals with $\phi(\xi)=\xi^3$,  C. Kenig, G. Ponce, and L. Vega \cite{Kenig_Ponce_Vega}, in connection to their seminal work on Korteweg-de Vries equations, 
established estimates of the form \eqref{main global maximal estimate} for $s>\frac{3}{4}$.\\

In \cite{Rog}, K. Rogers showed that in fact the local and global estimates \eqref{head local maximal estimate} and \eqref{main global maximal estimate} are equivalent in the following precise sense: if \eqref{head local maximal estimate} is valid for $s>s_0$ then \eqref{main global maximal estimate} is also valid for $s>as_0$, and vice-versa. The methods used in proving this result were based on a wave-packet analysis, which in a slightly different shape were used in Lee \cite{Lee} and T. Tao \cite{Tao} (for the Schr\"odinger maximal operators), and which ultimately stems from T. Wolff's paper \cite{Wolff}.\\

In this paper, we confine ourselves to the implication local to global and show that in this case, one can prove this, just using elementary methods based on simple scalings and classical estimates for oscillatory integrals. Thus no tools from the technical machinery of the wave-packet analysis are used.\\ 

Our main result is that the validity of \eqref{head local maximal estimate} for $s>s_0$ yields the validity of \eqref{main global maximal estimate} for $s>as_0$, for oscillatory integrals $T_t$ with $\phi$  positively homogeneous of degree $a$ with $a \geq 1$ (i.e. $\phi(r\xi)=r^a \phi(\xi)$, $r>0$), and satisfying
\begin{equation*} 
|\partial^{\alpha} \phi(\xi)|
	\lesssim |\xi|^{a-|\alpha|}, \quad \xi\in\R^n \setminus\{0\}\,\,\, \mathrm{and\,\,\, all\,\,\,  multi-indices}\,\,\, \alpha,
\end{equation*}
and
\begin{equation*}
\min_{|\xi|=1} |\nabla\phi(\xi)|>0.
\end{equation*}

Moreover this result is achieved via rather elementary means. Here it is important to mention that we actually manage to obtain endpoint results at all steps of the proof except the very last one, i.e. in the proof of Proposition \ref{LINLem:G_A<=>G}, which is the source of the \lq\lq $\varepsilon$-loss\rq\rq\, in the final conclusion.  However, we believe that removing the $\varepsilon$ behoves one to use other more advanced methods that won't fall into the scope of an elementary proof. \\

The paper is essentially self-contained and is organised as follows. In Section \ref{Sect:LIN} we use the Kolmogorov-Seliverstov-Plessner stopping time argument to \lq\lq linearise\rq\rq\, the problem and  show  in Theorem \ref{LINTh:L<=>G for phi} that local estimates yield global ones. The proof of Theorem \ref{LINTh:L<=>G for phi} is in turn divided into three propositions (Propositions \ref{LINLem:L<=>L_A}, \ref{Lem:Lee} and \ref{LINLem:G_A<=>G}). \\

In what follows, we shall omit all the constants that appear in various estimates, unless otherwise stated. In doing that we will use the notation $A\lesssim B$ which should be interpreted as $A\leq C B$ where $C$ is a constant. The dependence of $C$ on various other parameters will be clear from the context.

\section{local estimates imply global estimates} \label{Sect:LIN}

In what follows we shall denote by $H^s$ the Sobolev space of all tempered distributions $f$ for which $\langle \xi\rangle^ s\, \widehat{f}(\xi)\in L^2 (\R^n)$, where $\jap:=(1+|\xi|^2)^{1/2}.$ We shall also denote the Schwartz class by $\Ss(\R^n)$ and the class of smooth compactly supported functions by $C^{\infty}_{c} (\R^n)$.\\

We consider the operator
 \begin{equation*}
 T_t f(x)= \int_{\R^n} e^{ix\cdot\xi + it\phi(\xi)} \widehat{f}(\xi)\, d\xi,
 \end{equation*}
defined a-priori for $f\in \mathcal{S}(\R^n),$ where $\phi$ is a function that is positively homogeneous of degree $a$ with $a \geq 1$, and satisfies
\begin{equation} \label{condition of phase}
|\partial^{\alpha} \phi(\xi)|
	\lesssim |\xi|^{a-|\alpha|}, \quad \xi\in\R^n \setminus\{0\}\,\,\, \mathrm{and\,\,\, all\,\,\,  multi-indices}\,\,\, \alpha,
\end{equation}
and
\begin{equation}\label{curvature}
\min_{|\xi|=1} |\nabla\phi(\xi)|>0.
\end{equation}

The main goal of this paper is to establish the following result:

\begin{Th}\label{LINTh:L<=>G for phi in supnorm}
	Let $s_0 >0$, and $T_t$ be defined as above with the phase function satisfying \eqref{condition of phase} and \eqref{curvature}. Then the local bound
    \begin{equation*} 
    \| \sup_{0<t<1}|T_{t}f| \|_{L^2(B(0,1))}
    	\lesssim \|f\|_{H^s(\Rn)}, \quad s>s_0, \quad f \in \Ss(\R^n),
  \end{equation*}
   implies the global bound
    \begin{equation*} 
    \| \sup_{0<t<1}|T_{t}f| \|_{L^2(\Rn)}
    	\lesssim \|f\|_{H^{s}(\Rn)}, \quad s>as_0 , \quad f \in \Ss(\R^n).
\end{equation*}

\end{Th}

It is often more convenient to work instead with an equivalent \lq\lq linearized version\rq\rq\; of  the maximal operator given by
$$T_{t(x)}f(x)
	:= \int_{\Rn} e^{ix \cdot \xi +it(x)\, \phi(\xi)} \widehat{f}(\xi) d\xi, \quad x \in \Rn,$$
defined a-priori on Schwartz class of functions, for any  measurable function $0 \leq  t(x)  \leq  1.$
 Indeed it is well known that the linearized estimates imply sup-norm estimates by the classical Kolmogorov-Seliverstov-Plessner stopping time argument. On the other hand, trivially, for any measurable $t(x)\in [0,1]$ and any $f\in \Ss(\R^n)$ one has that for all $x$,
$$|T_{t(x)} f(x)|\leq \sup_{0<t<1} |T_{t} f(x)|.$$
Therefore, any norm estimate for the expression on the right hand side implies the one for term on the left hand side. Thus, from now on we shall put our efforts in proving the following theorem:

\begin{Th}\label{LINTh:L<=>G for phi}
	Let $s_0 >0$ and $0 \leq t(x)\leq  1$ be a measurable function. Then,  the linearized local bound
    \begin{equation*} 
    \| T_{t(x)}f \|_{L^2(B(0,1))}
    	\lesssim \|f\|_{H^s(\Rn)}, \quad s>s_0, \quad f \in \Ss(\R^n),
  \end{equation*}
    implies the linearized global bound
    \begin{equation*} 
    \| T_{t(x)}f \|_{L^2(\Rn)}
    	\lesssim \|f\|_{H^{s}(\Rn)}, \quad s>as_0, \quad f \in \Ss(\R^n).
\end{equation*}
\end{Th}

It is absolutely crucial to emphasize that in these estimates and all the forthcoming ones, the constants of the estimates are independent of the measurable functions that are involved in the definition of the operators.\\

The proof  of Theorem \ref{LINTh:L<=>G for phi} is divided in three steps, that we present below as separate results (Propositions \ref{LINLem:L<=>L_A}, \ref{Lem:Lee} and \ref{LINLem:G_A<=>G}).\\

At this point we shall introduce the space $\SA$ consisting of all those functions in the Schwartz class whose frequency is supported in the unit annulus; that is,
$$\SA:=\Big\{f \in \Ss(\Rn): \supp(\widehat{f}) \subset A(1) \Big\},$$
where $A(R):=\{R/2 \leq  |\xi| \leq 2 R\}$, $R>0$.
In the proofs of the next results
it will be crucial to use  the following partition of the unity. We start by choosing a radial function $\chi \in C_c^\infty(\R^n)$ such that $0 < \chi \leq 1$ in $B(0,2)$, $\chi \equiv 1$ in $B(0,1)$ and $\chi \equiv 0 $ in $\R^n \setminus B(0,2)$. Next, we set
\begin{equation}\label{eq:deflambda}
\lambda(\xi)
	:= \chi(\xi) - \chi(2\xi), \quad \xi  \in \Rn,
\end{equation}
which is radial and supported in the annulus $\{1/2 \leq  |\xi| \leq 2\}$ and does not vanish at any isolated point inside its support. Finally, we define
\begin{equation}\label{eq:partunity}
\psi_0(\xi):= \chi(\xi), \qquad
\psi_k(\xi):= \lambda(2^{-k} \xi), \quad k \geq 1.
\end{equation}
Observe that $\supp \psi_k \subset \{2^{k-1} \leq |\xi| \leq 2^{k+1}\}$, $k \geq 1$,
\begin{equation*}
	\supp \psi_k \cap \supp \psi_j = \emptyset, \quad |k-j|>1,
\end{equation*}
and
\begin{equation*}
	\sum_{k \geq 0} \psi_k(\xi) = 1, \quad \xi  \in \Rn.
\end{equation*}
In dealing with the low frequency portions of the oscillatory integral operators $T_{t(x)}f(x)$ the following lemma will be useful.

\begin{Lem}\label{Lem:lowfreq}
Assume that $t(x)$ is a measurable function with $0 \leq t(x)  \leq 1$, $\chi(\xi)\in C_{c}^{\infty}(\R^n)$ is a smooth cut-off function supported in a neighborhood of the origin, and let $\phi$ be a positively homogeneous of degree $a\geq 1$ phase function satisfying \eqref{condition of phase}. Consider the operator
$$R_{t(x)}f(x):= \int_{\Rn} \chi(\xi)\, e^{ix \cdot \xi +it(x)\, \phi(\xi)} \,  \widehat{f}(\xi)\, d\xi, \quad x \in \Rn.$$
Then for $1\leq p\leq \infty$ one has
\begin{equation*}
\Vert R_{t(x)}f\Vert_{L^p(\R^n)}\lesssim \Vert f\Vert_{L^p(\R^n)}.
\end{equation*}
\end{Lem}

\begin{proof}
Since
$$R_{t(x)}f(x)
=\int_{\R^n} K(x,y)\, f(y)\, dy$$
with
$$K(x,y)
:= \frac{1}{(2\pi)^n}\int_{\Rn} \chi (\xi)\, e^{i(x-y) \cdot \xi +it(x)\, \phi(\xi)}\, d\xi,$$
the result would follow from Schur's lemma, if we manage to show that $$\sup_{x \in \R^n} \Vert K(x, \cdot)\Vert_{L^1(\R^n)}<\infty
\quad \text{and} \quad
\sup_{y \in \R^n} \Vert K(\cdot,y )\Vert_{L^1(\R^n)}<\infty.$$

The proof is divided into two cases. First consider the case when $a$ (the degree of homogeneity of $\phi$) is equal to one. In this case we have for any multi-index $\alpha$ with $|\alpha|=n$ and $|\alpha|=n+1$
$$\sup_{\xi\in\R^n\setminus \{0\}} |\xi|^{|\alpha|-1}\Vert\partial_{\xi}^{\alpha}(e^{it(\cdot)\,  \phi(\xi) } \,\chi (\xi) )\Vert_{L^\infty (\Rn)}<+\infty.$$ Therefore \cite[Lemma 1.17]{DS} (actually its proof) yields that for all $0\leq \varepsilon< 1$ one has $|K(x,y)|\lesssim \langle x-y\rangle ^{-n-\varepsilon}$, where the hidden constant on the right hand side of this estimate doesn't depend on $t(x).$ This kernel estimate obviously implies the Schur-type estimates above.\\

For the case $a>1$ we claim that $|K (x,y)|\lesssim \langle x-y\rangle ^{-n-1},$ where once again the hidden constant on the right hand side of this estimate doesn't depend on $t(x).$ Since $|K(x,y)|\lesssim 1,$ it is enough to show that that $|K(x,y)|\lesssim|x-y|^{-n-1}$. To this end we split the kernel into $K(x,y)= K_{1}(x, y)+K_{2}(x, y)$ where
$$K_{1}(x,y):= \frac{1}{(2\pi)^n}\int_{\Rn} e^{i(x-y) \cdot \xi}\,  \chi(\xi)\, d\xi, $$
and
$$K_{2}(x,y):= \frac{1}{(2\pi)^n}\int_{\Rn} e^{i(x-y) \cdot \xi}\,  \chi (\xi)\, ( e^{it(x)\, \phi(\xi)}-1)\, d\xi.
$$
Since $\chi\in C^{\infty}_{c}(\R^n),$ we have that
$|K_{1}(x,y)|\lesssim | x-y |^{-N},$ for all $N\geq 0$.\\

Given $0<\delta<1$, we introduce a smooth function $\rho$ with $0\leq \rho\leq 1$ such that $\rho(\xi)=1$ when $|\xi|\geq 2$ and $\rho(\xi)=0$ when $|\xi|\leq 1.$ Now setting
$$K_{2,\delta}(x,y):=\frac{1}{(2\pi)^n}\int_{\Rn} e^{i(x-y) \cdot \xi}  \chi(\xi)\,\rho\Big(\frac{\xi}{\delta}\Big)\, ( e^{it(x)\, \phi(\xi)}-1)\, d\xi,$$
Lebesgue's dominated convergence theorem yields that $K_2(x,y)=\lim_{\delta\to 0} K_{2,\delta}(x,y)$.\\

If we integrate by parts $n+1$ times in the integral defining $K_{2,\delta}(x,y)$ we obtain
$$|K_{2,\delta} (x,y)|\lesssim |x-y|^{-n-1}\sum_{|\alpha|+|\beta|+|\gamma|=n+1}\delta^{-|\gamma|}\int_{\R^n} |\partial^{\alpha}( e^{it(x)\, \phi(\xi)}-1)|\,
|\partial^{\beta}\chi(\xi)| |(\partial^\gamma \rho)\Big(\frac{\xi}{\delta}\Big)|\, d\xi.$$

At this point we observe that by the conditions on $t(x)$ and $\phi,$ we have for all multi-indices $\alpha$ $|\partial^{\alpha}( e^{it(x)\, \phi(\xi)}-1)|\lesssim |\xi|^{a-|\alpha|}$ uniformly in $t(x),$ for $\xi$ in the support of $\chi$.
Therefore since $a>1$, if $\gamma=0$ then the corresponding term in the sum above is bounded by
$$|x-y|^{-n-1}\sum_{|\alpha|+|\beta|=n+1}\int_{\supp \chi}  |\xi|^{a-n-1}\, d\xi\lesssim |x-y|^{-n-1}.$$
On the other hand, those terms with $|\gamma| \geq 1$, are bounded by
\begin{align*}
& |x-y|^{-n-1}\sum_{|\alpha|+|\beta|+|\gamma|=n+1}\int_{\delta
\leq |\xi|\leq 2\delta}  |\xi|^{a-|\alpha|}\, \delta^{-|\gamma|}\, d\xi
 \lesssim  |x-y|^{-n-1} \delta^{a-1}.
\end{align*}
Taking the limit as $\delta$ goes to zero, we obtain
\[
	|{K_2(x,y)}|\lesssim |x-y|^{-n-1}.
\]

This establishes the desired kernel estimate, and once again Schur's lemma, enable us to deduce the $L^p$ boundedness of operator $R_{t(x)}$.
\end{proof}

\begin{Rem}
As a matter of fact, the case $a>1$ could also be dealt with, following the same argument as in the case of $a=1.$ However, since the argument presented above, which is similar to that in \cite{Sjolin1}, yields a better decay, we provided a separate proof  in order to maintain a more self-contained presentation.
\end{Rem}

For our forthcoming estimates we would also need the following version of the non-stationary phase lemma, whose proof can be found in \cite[Lemma 3.2]{RLS}.

\begin{Lem}\label{LINlem:technic}
Let $\mathcal{K}\subset \mathbb{R}^n$ be a compact set and  $U\supset \mathcal{K}$ an open set. Assume that $\Phi$ is a real valued function in $C^{\infty}(U)$ such that $|\nabla \Phi|>0$ and
$$|\d^\alpha \Phi| \lesssim |\nabla \Phi|,\qquad
|\d^\alpha (|\nabla \Phi|^2)|\lesssim |\nabla \Phi|^2,$$
for all multi-indices $\alpha$ with $|\alpha|\geq 1$.
Then, for any $F\in C^\infty_c (\mathcal{K})$ and any integer $k\geq 0$,
\begin{equation*}
	 \Big| \int_{\R ^n} F(\xi)\, e^{i\Phi(\xi)}\, d \xi \Big|
     \leq C_{k,n,\mathcal{K}} \sum_{|\alpha| \leq k} \int_\mathcal{K} |\d^{\alpha} F(\xi)| \, |\nabla \Phi(\xi)|^{-k}\, d \xi.
\end{equation*}
\end{Lem}

Now we shall proceed with our chain of propositions.

\begin{Prop}\label{LINLem:L<=>L_A}
For $s>0$, if  for all measurable functions $0 \leq t(x) \leq 1$, the estimate
    \begin{equation}\label{ball of radius 1 sobolev}
    \| T_{t(x)}f  \|_{L^2(B(0,1))}
    	\lesssim \|f\|_{H^s(\Rn)}, \quad f \in \Ss(\R^n),
\end{equation}
 holds, then one has
\begin{equation*}
\| T_{\tau(x)}f  \|_{L^2(B(0,R))}
    	\lesssim R^s \|f\|_{L^2(\Rn)}, \quad f \in \SA, \ R\geq 1,
\end{equation*}
    for all measurable functions $0 \leq \tau(x) \leq  R^a$.
\end{Prop}

\begin{proof}
Let $f \in \SA$, $R\geq 1$ and $0\leq  \tau(x) \leq  R^a$ a measurable function. Take $t(x):=\tau(x)/R^a$. A change of variables yields
    \begin{align*}
    	T_{\tau(x)}f(x)
           & = R^{-n}  \int_{\Rn} e^{i\frac{x}{R}\eta + i t(x)\, \phi(\eta)} \widehat{f}\Big(\frac{\eta}{R}\Big) d\eta.
    \end{align*}
Setting $f_{R}(z):=f(Rz)$ and using \eqref{ball of radius 1 sobolev} it follows that
    \begin{align}\label{eq:Trho2}
    	\| T_{\tau(x)}f \|_{L^2(B(0,R))}
            & = R^{n/2} \| T_{t(Rx)} f_R \|_{L^2(B(0,1))}
             \lesssim R^{n/2} \|f_R\|_{H^s(\Rn)} \nonumber
              \lesssim  R^{s}  \|f\|_{L^2(\Rn)},
    \end{align}
   because $\supp(\widehat{f_R}) \subset A(R)$ and
   $0\leq  t(Rx) =\tau(Rx)/R^a \leq  1$, $x \in \Rn$.
\end{proof}

The following proposition gives us a means of transferring local to global estimates for frequency localised functions.

\begin{Prop}\label{Lem:Lee}
	For $s>0$, if for all measurable functions $0 \leq \tau(x) \leq R^a$, the estimate
\begin{equation}\label{eq:tau}
    \| T_{\tau(x)}f  \|_{L^2(B(0,R))}
 \lesssim R^s \|f\|_{L^2(\Rn)}, \quad f \in \SA, \, R\geq 1,
\end{equation}
 holds, then one has
\begin{equation}\label{eq:rho}
    \| T_{\rho(x)}f \|_{L^2(\Rn)}
    	\lesssim R^{as} \|f\|_{L^2(\Rn)}, \quad f \in \SA,  \, R\geq 1,
\end{equation}
for all measurable functions $0 \leq  \rho(x) \leq R^a$.
\end{Prop}

Of course the two estimates above are the same when $a=1$, so we can confine ourselves to the cases $a>1$.

\begin{proof}
First observe that \eqref{eq:tau} trivially yields that for $0\leq \tau (x)\leq R$ one has
\begin{equation}\label{LINlocal simple go to R}
    \| T_{\tau(x)}f  \|_{L^2(B(0,R))}
 \lesssim R^s \|f\|_{L^2(\Rn)}, \quad f \in \SA, \, R\geq 1.
\end{equation}

Let $\theta$ be a smooth function that is equal to one on $A(1)$ and supported in $\{1/4<|\xi|<4\}.$
We partition $\R^n$ into finitely overlapping balls $\{B(x_j, R^a)\}_{j \in \Z}$. Let $M:=\sup_{|{\xi}|=1} |\nabla\phi(\xi)|$ and set $\kappa:=4^{a}M$. Then
\begin{align}\label{eq-33}
 \| T_{\rho(x)} f \|_{L^2(\Rn)}^2
&\lesssim \sum_{j \in \Z} \int_{B(x_j,R^a)} \Big| \int_{\Rn}\int_{\Rn} e^{i(x-y)\cdot \xi+i\rho(x)\,  \phi(\xi)}\, \theta(\xi)\, f(y)\,d\xi\, dy \Big|^2 dx  \nonumber \\
	&\lesssim\sum_{j \in \Z} \int_{B(x_j,R^a)}\Big| \int_{\Rn} \int_{\Rn} e^{i(x-y)\cdot \xi+i\rho(x)\,  \phi(\xi)}\, \theta(\xi)\, \psi_{j,R} (y)\, f(y)\,d\xi\, dy \Big|^2\, dx\,  \nonumber \\
	& \qquad  + \sum_{j \in \Z} \int_{B(x_j,R^a)} \Big| \int_{\Rn} \int_{\Rn} e^{i(x-y)\cdot \xi+i\rho(x)\,  \phi(\xi)}\, \theta(\xi)\, \Big(1-\psi_{j,R} (y)\Big)\, f(y)\,d\xi\, dy \Big|^2 dx,
\end{align}
where $\psi_{j,R}$ is a bump function equal to $1$ on the ball $B(x_j, (\kappa+1)R^a)$ and supported in the ball $B(x_j, (\kappa+2)R^a)$.
For the first term above, we decompose $\theta$ as
$$\theta =: \theta_1 +\theta_2 + \theta_3,$$
where $\theta_1$ is supported in $\{1/4\leq |\xi|\leq 1\},$ $\theta_2$ is supported in $\{1/2\leq|\xi|\leq 2\}$ and $\theta_3$ is supported in $\{1\leq |\xi|\leq 4\}$.
For instance, we could take
$\theta_1(\xi):=\lambda(2\xi)$,
$\theta_2(\xi):=\lambda(\xi)$
and $\theta_3(\xi):=\lambda(\xi/2)$,
where $\lambda$ is the function introduced in \eqref{eq:deflambda}.\\

Then we have that
\begin{align*}
& \sum_{j \in \Z} \int_{B(x_j,R^a)}\Big| \int_{\Rn} \int_{\Rn} e^{i(x-y)\cdot \xi+i\rho(x)\,  \phi(\xi)}\, \theta(\xi)\, \psi_{j,R} (y)\, f(y)\,d\xi\, dy \Big|^2\, dx \\
& \qquad \lesssim \sum_{k=1}^{3} \sum_{j \in \Z} \int_{B(x_j,R^a)}\Big| \int_{\Rn} \int_{\Rn} e^{i(x-y)\cdot \xi+i\rho(x)\,  \phi(\xi)}\, \theta_{k} (\xi)\, \psi_{j,R} (y)\, f(y)\,d\xi\, dy \Big|^2\, dx.
\end{align*}
We first analyze the term that contains $\theta_2$ since $\supp \theta_2\subset A(1)$. Using the fact that $0 \leq  \rho(x+x_j) \leq  R^a$, and setting $\widehat{g_{2,j,R}}(\xi):= \theta_{2}(\xi) \widehat{\psi_{j,R} f}(\xi)$, $\tau_h f(x):= f(x+h)$, estimate \eqref{LINlocal simple go to R} yields
\begin{align}\label{eq:theta2}
	& \sum_{j \in \Z}  \int_{B(x_j,R^a)} \Big| \int_{\Rn} \Big\{\int_{\Rn} e^{i(x-y)\cdot \xi+i\rho(x)\,  \phi(\xi)}\, \theta_{2} (\xi)\,d\xi\Big\} \psi_{j,R} (y)\, f(y)\, dy \Big|^2\, dx   \nonumber \\
	& \qquad \qquad = \sum_{j \in \Z}  \int_{B(x_j,R^a)} \Big| \int_{\Rn} e^{ix\cdot \xi+i\rho(x)\,  \phi(\xi)} \widehat{g_{2,j,R}}(\xi) \, d\xi \Big|^2\, dx \nonumber \\
    & \qquad \qquad = \sum_{j \in \Z}  \int_{B(x_j,R^a)} |T_{\rho(x)}g_{2,j,R} (x)|^2\, dx
     = \sum_{j \in \Z}  \int_{B(0,R^a)} |T_{\rho(x+x_j)}(\tau_{x_j} g_{2,j,R})(x)|^2 \, dx \nonumber \\
    & \qquad \qquad \lesssim R^{2as} \sum_{j \in \Z} \Big\|\tau_{x_j} g_{2,j,R} \Big\|^2_{L^2(\Rn)}
    \lesssim R^{2as} \|f\|_{L^2(\Rn)}^2,
\end{align}
where in the last estimate, we have used the translation invariance of the $L^2$ norm,  Plancherel's formula and the finite overlapping property of the dilations of the supports.\\

To deal with the integral containing $\theta_1$, we set $\widehat{g_{1,j,R}}(\xi):= \theta_{1}(\xi/2) \widehat{\psi_{j,R} f}(\xi/2),$  and follow a similar line of calculations as in the case of $\theta_2$, with the difference that here we make changes of variables and use the homogeneity of $\phi$. This leads to
\begin{align}\label{eq:theta1}
	 \sum_{j \in \Z} & \int_{B(x_j,R^a)} \Big| \int_{\Rn} \Big\{\int_{\Rn} e^{i(x-y)\cdot \xi+i\rho(x)\,  \phi(\xi)}\, \theta_{1} (\xi)\,d\xi\Big\} \psi_{j,R} (y)\, f(y)\, dy \Big|^2\, dx \nonumber \\
&= 2^{-n} \sum_{j \in \Z}  \int_{B(0,\frac{R^a}{2})} \Big|\int_{\Rn}  e^{ix\cdot \xi+i\frac{\rho(2x+x_j)}{2^a}\,  \phi(\xi)} e^{i\frac{x_j}{2}\cdot \xi}\, \widehat{g_{1,j,R}}(\xi) \, d\xi \Big|^2 \,dx \nonumber \\
  &   \leq  2^{-n} \sum_{j \in \Z} | \int_{B(0,R^a)}| T_{\frac{\rho(2x+x_j)}{2^a}}(\tau_{\frac{x_j}{2}}\, g_{1,j,R})(x)|^2 \, dx  \nonumber \\
&   \lesssim 2^{-n} R^{2as} \sum_{j \in \Z} \Big\|\tau_{x_j/2}\, g_{1,j,R} \Big\|^2_{L^2(\Rn)}
  \lesssim R^{2as} \|f\|_{L^2(\Rn)}^2,
\end{align}
where we have used the facts that $\supp\widehat{g_{1,j,R}}\subset A(1)$ and $0 \leq  \rho(2x+x_j)/2^a  \leq  R^{a}/2^a < R^a$.\\

To deal with the integral containing $\theta_3$, we set $\widehat{g_{3,j,R}}(\xi):= \theta_{3}(2\xi) \widehat{\psi_{j,R} f}(2\xi),$  and once again use a suitable change of variables and the homogeneity of $\phi.$ This yields
\begin{align} \label{eq:theta3}
& \sum_{j \in \Z}  \int_{B(x_j,R^a)} \Big| \int_{\Rn} \Big\{\int_{\Rn} e^{i(x-y)\cdot \xi+i\rho(x)\,\phi(\xi)}\,\theta_{3} (\xi)\,d\xi\Big\} \psi_{j,R} (y)\, f(y)\, dy \Big|^2 dx   \nonumber \\
& \qquad \qquad = \sum_{j \in \Z} 2^{n} \int_{B(0,2 R^a)} \Big|\int_{\Rn}  e^{ix\cdot \xi+i2^a \rho(\frac{x}{2}+x_j)\,  \phi(\xi)} e^{2ix_j \cdot \xi}\, \widehat{g_{3,j,R}}(\xi) \, d\xi \Big|^2 dx \nonumber \\ & \qquad \qquad  \leq  2^{n} \sum_{j \in \Z}  \int_{B(0,\,(2 R)^{a})}| T_{2^a \rho(\frac{x}{2}+x_j)}(\tau_{2x_j}\, g_{3,j,R})(x)|^2\, dx \nonumber  \\
    & \qquad \qquad \lesssim 2^{n} 2^{2as} R^{2as} \sum_{j \in \Z} \Big\|\tau_{x_j/2}\, g_{3,j,R} \Big\|^2_{L^2(\Rn)}
    \lesssim R^{2as} \|f\|_{L^2(\Rn)}^2,
\end{align}
where we have used the facts that $\supp\widehat{g_{3,j,R}}\subset A(1)$ and $0 \leq 2^a \rho(x/2+ x_j ) \leq (2R)^{a}.$\\

To estimate the term containing $1-\psi_{j,R}$ in \eqref{eq-33}, we set $F:=\theta(\xi)$, $\Phi:= (x~-~y)~\cdot~\xi~+~\rho(x)\,  \phi(\xi)$, and observe that $\nabla_{\xi}\Phi= x-y + \rho(x)\nabla \,  \phi(\xi)$ verifies all the assumptions of Lemma \ref{LINlem:technic}. Indeed, as a first step we have that, for $0\leq  \rho(x) \leq  R^a$, $|y-x_j|\geq (\kappa+1) R^a$ and $x\in B(x_j,R^a)$, the estimate $|x-y|\geq \kappa R^a \geq \kappa \rho(x)$ holds true.\\

Now define $m:=\min_{|{\xi}|=1} |\nabla\phi(\xi)|$, and observe that $m>0$ by the assumption \eqref{curvature} on the phase. We claim that
\begin{equation}\label{nabla_estimate}
	|\nabla_\xi \Phi(\xi)|\geq \max\brkt{{\frac{3|{x-y}|}{4},{3\rho(x)|\nabla \phi(\xi)|}}}
    \geq\max\brkt{{\frac{3|{x-y}|}{4},3m\rho(x)|{\xi}|^{a-1}}},
\end{equation}
where the second lower bounds above follow from  the homogeneity of $\phi$. Therefore it remains to prove the first lower bounds. To this end, we have for $\xi\in\mathrm{supp}\,\theta$ i.e. for $\{1/4<|\xi|<4\},$
\begin{equation}\label{eq:easybound}
	|\nabla \phi(\xi)|\leq M |\xi|^{a-1} \leq  4^{a-1} M =\frac{\kappa}{4}.
\end{equation}
Thus,
$$|\nabla_{\xi}\Phi(\xi)|\ge |x-y|- \rho(x)\,|\nabla \phi(\xi)| > |x-y|- \frac{|x-y|}{\kappa} \frac{\kappa}{4} 
=\frac{3|x-y|}{4}.$$
Now since $|x-y|> \kappa \rho(x),$ we have
\[
	|\nabla_\xi \Phi(\xi)|\geq |x-y|-\rho(x){|\nabla \phi(\xi)|}\geq
    \rho(x){|\nabla \phi(\xi)|}\brkt{\frac{\kappa}{|\nabla \phi(\xi)|}-1}.
\]
Moreover, \eqref{eq:easybound} implies that
\begin{equation}\label{lower bound for grad phi}
	|\nabla_\xi \Phi(\xi)|\geq 3\, \rho(x)\,{|\nabla \phi(\xi)|}\geq 3\, m\, \rho(x)\,|\xi|^{a-1}.
\end{equation}
Trivially, for any $|\alpha|=1$, $|\partial^\alpha_\xi \Phi(\xi)|\leq |\nabla_\xi \Phi(\xi)|$.\\

For $|\alpha|\geq 2$ and $\{1/4<|\xi|<4\},$ \eqref{condition of phase} and \eqref{lower bound for grad phi} imply that
\[
	\begin{split}
	|\partial_\xi^\alpha \Phi(\xi)| &= \rho(x)|\partial^\alpha \phi(\xi)|\leq c_\alpha
    \rho(x)|{\xi}|^{a-1+1-|{\alpha}|}=c_\alpha 4^{|{\alpha}|-1}
    \rho(x)|{\xi}|^{a-1}\\
    &\leq \frac{1}{3m} c_\alpha 4^{|{\alpha}|-1}  |\nabla_\xi \Phi(\xi)|\lesssim |\nabla_\xi \Phi(\xi)| ,
	\end{split}
\]
which verifies the first condition (on the phase) of Lemma \ref{LINlem:technic}.
To check the validity of the second condition on the phase in Lemma \ref{LINlem:technic}, we observe that since
\[
	|\nabla_\xi \Phi(\xi)|^2=|x-y|^2+\rho(x)^2|\nabla\phi(\xi)|^2+2\rho(x)(x-y)\cdot \nabla\phi(\xi),
\]
we have that for any $|\alpha|\geq 1$,
\begin{equation}\label{alpha derivative of gradient}
	\partial^\alpha_\xi |\nabla_\xi \Phi(\xi)|^2=\rho(x)^2 \partial^\alpha_\xi|\nabla\phi(\xi)|^2+2\rho(x)(x-y)\cdot \nabla (\partial^\alpha\phi)(\xi).
\end{equation}
For the second term on the RHS of \eqref{alpha derivative of gradient}, estimate \eqref{lower bound for grad phi} yields
\[
	2|\rho(x)(x-y)\cdot \nabla (\partial^\alpha\phi)(\xi)| \leq 2 c_\alpha \rho(x) |x-y| |\xi|^{a-|\alpha|-1}\leq \frac{2}{3 m} c_\alpha 4^{|\alpha|}   |x-y| |\nabla_\xi \Phi(\xi)|\\
    \lesssim  |\nabla_\xi \Phi(\xi)|^2,
\]
where the last inequality follows from \eqref{nabla_estimate}. For the first term on the RHS of \eqref{alpha derivative of gradient}, Leibniz's rule and equation \eqref{nabla_estimate} yield
\[
	\rho(x)^2|\partial^\alpha_{\xi} |\nabla\phi(\xi)|^2|\lesssim \rho(x)^2|{\xi}|^{2a-2-|\alpha|}\lesssim |\nabla_\xi \Phi(\xi)|^2.
\]
Therefore, Lemma \ref{LINlem:technic} implies that for $0<\rho(x)<R^a$, $|y-x_j|>(\kappa +1) R^a$,  $x\in B(x_j,R^a)$ and all $N\geq 0$
\begin{equation*} 
\Big|\int_{\Rn} e^{i(x-y)\cdot \xi+i\rho(x)\,  \phi(\xi)}\, \theta(\xi)\,d\xi\Big|
	\lesssim R^{-aN}(1+|x-y|)^{-N}.
\end{equation*}
Now if $Mf(x)$ denotes the Hardy-Littlewood maximal function of $f$, then for any $N\geq 0$ one has
\begin{align}\label{LINeq-36}
	& \sum_{j \in \Z} \int_{B(x_j,R^a)} \Big| \int_{\Rn} \Big\{\int_{\Rn} e^{i(x-y)\cdot \xi+i\rho(x)\,  \phi(\xi)}\, \theta(\xi)\,d\xi \Big\} \Big(1- \psi_{j,R} (y)\Big)\, f(y)\, dy \Big|^2 dx \nonumber \\
	& \qquad \qquad \lesssim \sum_{j \in \Z} \int_{B(x_j,R^a)} \Big(\int_{\Rn}  R^{-aN}(1+|x-y|)^{-N} |f(y)|\, dy \Big)^2\, dx \nonumber \\
    & \qquad \qquad \lesssim   R^{-2aN}\sum_{j \in \Z} \int_{B(x_j,R^a)}|Mf(x)|^2\, dx
     \lesssim  R^{-2aN}\Vert Mf\Vert^{2}_{L^2(\R^n)}
    \lesssim  R^{-2aN}\Vert f\Vert^{2}_{L^2(\R^n)}.
\end{align}

Finally putting \eqref{eq:theta2}, \eqref{eq:theta1}, \eqref{eq:theta3} and \eqref{LINeq-36} together we obtain \eqref{eq:rho}.
\end{proof}
The last step in the chain of propositions which establishes Theorem \ref{LINTh:L<=>G for phi} is the following:


\begin{Prop}\label{LINLem:G_A<=>G}
	Let $\eps,s>0$. If   
    \begin{equation} \label{LINglobal double}
    \| T_{\rho(x)}f \|_{L^2(\Rn)}
    	\lesssim R^{as} \|f\|_{L^2(\Rn)}, \quad f \in \SA,  \, R\geq 1,
    \end{equation}
for all measurable functions $0 \leq \rho(x) \leq R^a$, then
    \begin{equation*}
    \| T_{t(x)}f \|_{L^2(\Rn)}
    	\lesssim \|f\|_{H^{as+\eps}(\Rn)}, \quad f \in \Ss(\R^n),
\end{equation*}   
for all measurable functions $t(x)$ with $0 \leq  t(x) \leq  1$.
\end{Prop}

\begin{proof}
Let $f \in\mathcal{S}(\R^n)$ and $0\leq t(x)\leq 1$ measurable.
It is enough to prove that
\begin{equation*}
\| \T_{t(x)} f \|_{L^2(\R^n)}
	\lesssim \| f \|_{L^2(\R^n)},
\end{equation*}
where
\begin{equation*}\label{defn:double T}
\T_{t(x)} f(x)
	:= \int_{\Rn}  \jap^{-(as+\eps)} e^{ix \cdot \xi + it(x)\,  \phi(\xi)} \widehat{f}(\xi) d\xi, \quad x \in \Rn.
\end{equation*}

In association to the partition of the unity defined in \eqref{eq:partunity} we consider the standard Littlewood-Paley decomposition,
$$f=\sum_{k \geq 0} \PP_k f, \qquad
   \widehat{\PP_k f}(\xi)= \psi_k (\xi)\,\widehat{f}(\xi), \ k \geq 0.$$
By the triangle inequality we can write
\begin{equation}\label{eq:triangineq}
\| \T_{t(x)} f \|_{L^2(\R^n)}
\leq  \| \T_{t(x)} (\PP_0 f) \|_{L^2(\R^n)} 
+ \sum_{k \geq 1} \| \T_{t(x)} (\PP_k f) \|_{L^2(\R^n)}.
\end{equation}
First we claim that
\begin{equation}\label{2localized result}
\Vert \T_{t(x)} (\PP_0 f) \Vert_{L^2(\R^n)}
	\lesssim \Vert f\Vert_{L^2(\R^n)}.
\end{equation}
To see this we just observe that the integral kernel of $\T_{t(x)} (\PP_0 f)$ is given by
\begin{equation*}
 \int_{\Rn} \jap^{-(as+\eps)}\, \psi_0 (\xi) \, e^{i(x-y) \cdot \xi + it(x)\,  \phi(\xi) }   \,  d\xi,
\end{equation*}
to which the kernel estimate of Lemma \ref{Lem:lowfreq} is applicable.\\

Second, in order to be able to use assumption \eqref{LINglobal double}, we observe that if  $g \in \Ss(\R^n)$ with $\supp(\widehat{g}) \subset A(R)$. Taking $\rho(x):=R^a t(x)$ and changing variables yield
\begin{align} \label{LINeq:homog}
	\T_{t(x)} g(x)
     & = \int_{\Rn} e^{ix \cdot \xi + i \rho(x)\frac{\phi(\xi)}{R^a}} \widehat{g}(\xi) \jap^{-(as+\eps)}\, d\xi \nonumber \\
    & = R^{n} \int_{\Rn} e^{i R x \cdot\xi + i \rho(x)\, \phi(\xi)} \widehat{g}(R \xi) \langle R \xi \rangle^{-(as+\eps)}\, d\xi.
\end{align}
 Define $\widehat{h_{1/R}}(\eta):= \widehat{g_{1/R}}(\eta) \langle R \eta \rangle^{-(as+\eps)},$ where  $g_{1/R}(z):=g(z/R)$ and observe that $\supp(\widehat{h_{1/R}})=\supp(\widehat{g_{1/R}}) \subset A(1)$.
    Therefore, \eqref{LINeq:homog} and  \eqref{LINglobal double} give us
  \begin{align} \label{LINeq:Problem}
\| \T_{t(x)} g \|_{L^2(\Rn)}
	&= R^{-n/2} \Big(\int_{\R^n} \Big| \int_{\R^n}
    e^{ix\cdot\xi + i\rho(x/R) \phi(\xi)}\langle R\xi\rangle ^{-(as+\eps)} \widehat{g}(R\xi) R^n  \, d\xi \Big|^2\, dx\Big)^{1/2}  \nonumber \\
    & \lesssim R^{as} R^{-n/2}  \|h_{1/R}\|_{L^2(\Rn)}
      = R^{as} R^{-n/2} \|\widehat{h_{1/R}}\|_{L^2(\Rn)} \nonumber\\
    & = R^{as} R^{-n/2} \Big(\int_{\Rn} |\widehat{g_{1/R}}(\eta) |^2 \, R^{-2(as+\eps)} (R^{-2}+ |\eta|^2)^{-(as+\eps)}\, d\eta \Big)^{1/2}   \nonumber \\
    & \lesssim R^{-\eps}
    R^{-n/2}  \|\widehat{g_{1/R}}\|_{L^2(\Rn)}
    = R^{-\eps}\|g\|_{L^2(\Rn)} ,
    \end{align}
   where we have also used the fact that $0 \leq \rho(x/R)= R^a \,t(x/R) \leq  R^a$, $x \in \Rn$.\\
   
Finally, putting together \eqref{eq:triangineq}, \eqref{2localized result} and taking $g=\PP_k f$, $R=2^k$, $k \geq 1$, in \eqref{LINeq:Problem} we conclude 
$$\| \T_{t(x)} f \|_{L^2(\R^n)}
\lesssim \| f \|_{L^2(\R^n)} + \sum_{k \geq 1} 2^{-k \eps} \| \PP_k f\|_{L^2(\R^n)}
\lesssim \| f \|_{L^2(\R^n)}.$$
\end{proof}

We should also mention that when the homogeneity degree $a$ of the phase $\phi$ is lager than $1$, then it is possible to prove Theorem \ref{LINTh:L<=>G for phi in supnorm}  for phases that verify the two conditions $$|\partial^{\alpha}\phi(\xi)|\lesssim |\xi|^{a-|\alpha|}
\quad \text{and} \quad 
|\nabla\phi(\xi)|\gtrsim |\xi|^{a-1}, \quad \text{for} \quad |\alpha|\leq 2
\quad \text{and} \quad |\xi|\neq 0.$$ 
For $a=1,$ one can replace these two conditions by $$|\partial^{\alpha}\phi(\xi)|\lesssim |\xi|^{1-|\alpha|}
\quad \text{for} \quad |\alpha|\leq 2
\quad \text{and} \quad |\xi|\neq 0$$ 
(e.g. the case of the Klein-Gordon equation).\\
 Though, for the sake of clarity and brevity of the exposition, we will not pursue these generalizations here and the details for the modifications of our arguments for inhomogeneous phases will appear elsewhere.


\end{document}